\documentclass[11pt,twoside]{amsart}  
\usepackage{amssymb}  
  
\numberwithin{equation}{section}  
\newtheorem{theorem}{Theorem}[section]  
\newtheorem{lemma}[theorem]{Lemma}  
\newtheorem{proposition}[theorem]{Proposition}  
\newtheorem{corollary}[theorem]{Corollary}

\theoremstyle{definition}  
\newtheorem{definition}[theorem]{Definition}
\newtheorem{remark}[theorem]{Remark}  
\newtheorem{example}[theorem]{Example}

\newtheorem{question}[theorem]{Question}

\DeclareMathOperator{\Tor}{Tor}  
\DeclareMathOperator{\reg}{reg}

\DeclareMathOperator{\lcm}{lcm}

\newcommand{\N}{\mathbb{N}}

\newcommand{\G}{\mathcal{G}}

\begin{document}  
    
   
\title{Splittings of monomial ideals}  
  
\author{Christopher A. Francisco}  
\address{Department of Mathematics, Oklahoma State University,   
401 Mathematical Sciences, Stillwater, OK 74078}  
\email{chris@math.okstate.edu}  
\urladdr{http://www.math.okstate.edu/$\sim$chris}  
  
\author{Huy T\`ai H\`a}  
\address{Tulane University \\ Department of Mathematics \\  
6823 St. Charles Ave. \\ New Orleans, LA 70118, USA}  
\email{tai@math.tulane.edu}  
\urladdr{http://www.math.tulane.edu/$\sim$tai/}  
  
\author{Adam Van Tuyl}  
\address{Department of Mathematical Sciences \\  
Lakehead University \\  
Thunder Bay, ON P7B 5E1, Canada}  
\email{avantuyl@lakeheadu.ca}  
\urladdr{http://flash.lakeheadu.ca/$\sim$avantuyl/}  
  
\thanks{Version: \today}  

\begin{abstract}
We provide some new conditions under which the graded Betti numbers of a monomial 
ideal can be computed in terms of the graded Betti numbers of smaller ideals, thus complementing 
Eliahou and Kervaire's splitting approach.  As applications, we show that edge ideals of graphs
 are splittable, and we provide an iterative method for computing the Betti numbers of the 
cover ideals of Cohen-Macaulay bipartite graphs. Finally, we consider the frequency with which
 one can find particular splittings of monomial ideals and raise questions about ideals 
whose resolutions are characteristic-dependent.
\end{abstract}
   
\maketitle  
  
  
\section{Introduction}  

The existence of computer algebra systems like CoCoA\cite{C} and Macaulay 2 \cite{M2} has made 
it easy to compute minimal free resolutions of ideals over $R=k[x_1,\dots,x_n]$, where $k$ is a 
field. However, we still have no closed formulas for the graded Betti numbers of arbitrary 
monomial ideals like we do in the special cases of stable ideals and complete intersections. 
One natural method for computing Betti numbers of a monomial ideal $I$ is to break $I$ down 
into smaller monomial ideals $J$ and $K$, where $I=J+K$, and the set of minimal generators of 
$I$ is the disjoint union of the minimal generators of $J$ and $K$.

Taking this approach in \cite{EK}, Eliahou and Kervaire introduced the notion
of splitting a monomial ideal.
Let $J$ and $K$ be monomial ideals such that 
$\G(I)$, the unique set of minimal generators of $I$, is the disjoint union of $\G(J)$ and $\G(K)$. Then $I=J+K$ is an \textbf{Eliahou-Kervaire splitting} (abbreviated as ``E-K splitting'') if 
there exits a splitting function \[ \G(J \cap K) \to \G(J) \times \G(K) \] sending 
$w \mapsto (\phi(w),\psi(w))$ such that 

\begin{enumerate}
\item $w=\lcm(\phi(w),\psi(w))$ for all $w \in \G(J \cap K)$, and
\item for every subset $S \subset \G(J \cap K)$, $\lcm(\phi(S))$ and $\lcm(\psi(S))$ 
strictly divide $\lcm(S)$.
\end{enumerate}

When $I=J+K$ is an E-K splitting, Eliahou and Kervaire proved in \cite[Proposition 3.1]{EK} 
that 
\begin{equation*} \label{splitformula} 
(\star)~~\hspace{.5cm}
\beta_{i,j}(I) = \beta_{i,j}(J)+\beta_{i,j}(K)+\beta_{i-1,j}(J \cap K),
\end{equation*} 
where $\beta_{i,j}(I)=\dim_k \Tor_i(k,I)_j$
is the $i,j$-th graded Betti number. 
Eliahou and Kervaire actually just proved $(\star)$ for total Betti numbers.  Fatabbi 
\cite[Proposition 3.2]{Fatabbi} extended the argument to the graded case; in fact, her proof works just as well if $j$ is a multidegree. 

E-K splittings have been used in a variety of contexts. Eliahou and Kervaire used 
them to study the Betti numbers of stable ideals \cite[Section 3]{EK}. Fatabbi \cite{Fatabbi},
Valla \cite{Va}, and the first author \cite{Francisco} used
E-K splittings to yield results on the graded Betti 
numbers of some ideals of fat points. The second and third authors used E-K splittings 
extensively to investigate the resolutions of edge ideals of graphs and hypergraphs
(see \cite{HVTsplit,HVThyper}). 

A substantial obstacle in using E-K splittings, however, is that it can be 
difficult to construct the required splitting function, or even to tell whether such a function exists. Our paper was motivated by a 
simple example in Eliahou and Kervaire's paper \cite[Remark 2]{EK}
(see also our Example \ref{e.ek}). They note that if 
$S=k[x_1,\dots,x_5]$, and \[ I = (x_1x_2x_3,x_1x_3x_5,x_1x_4x_5,x_2x_3x_4,x_2x_4x_5),\] then 
there is no E-K splitting of $I$. However, there are many ways to partition 
the minimal generators of $I$ to form smaller ideals $J$ and $K$ so that the formula
$(\star)$ still holds.

This example suggests that there are other conditions on $I$, $J$, and $K$,  beyond 
the criterion of Eliahou
and Kervaire,
that imply that formula $(\star)$ holds.  In fact, we wish to axiomatize this behavior
by introducing the following definition:

\begin{definition}
Let $I, J,$ and $K$ be monomial ideals such that $\G(I)$ is the disjoint union of $\G(J)$ and $G(K)$.
Then $I = J+K$ is a {\bf Betti splitting} if 
\[\beta_{i,j}(I) = \beta_{i,j}(J)+\beta_{i,j}(K)+\beta_{i-1,j}(J \cap K) \hspace{.5cm}~~\mbox{for all $i\in \N$
and (multi)degrees $j$.}\]
\end{definition}

The goal of this paper is to understand when a monomial ideal has a Betti splitting.
Such conditions would enable us to study the graded Betti numbers of more 
monomial ideals. The approach 
of splitting monomial ideals assumes that we know some information about the minimal 
resolutions of $J$, $K$, and $J \cap K$, and thus it is natural to investigate conditions on 
the Betti numbers of those ideals that force $I=J+K$ to be a Betti splitting. Our focus is on 
constructing $J$ and $K$ so that their resolutions have little ``overlap'' with that of 
$J \cap K$.   Working with multigraded Betti numbers, as opposed to the total Betti numbers as 
in \cite{EK}, actually simplifies some of our arguments and enables us to prove stronger results than we could even with graded Betti numbers.

We begin in Section 2 by showing that Betti splittings are intimately related
to maps between Tor modules;  we find some sufficient conditions for Betti splittings, and compare the 
applicability of our results to those of Eliahou and Kervaire. In Section 3, we apply
our approach to ideals associated to graphs and hypergraphs. In particular, 
we give a very short proof that edge ideals of graphs can be split in a canonical way. In 
addition, we develop an iterative method of computing the graded Betti numbers of cover ideals 
of Cohen-Macaulay bipartite graphs. Resolving cover ideals of graphs is generally a difficult task because simply to compute the minimal generators, one has to find all minimal vertex covers of the graph, which is a NP-complete problem. We conclude in Section 4 by commenting on the ubiquity of Betti splittings that becomes clear from computational experiments in Macaulay 2 and some 
interesting cases of ideals whose resolutions are characteristic-dependent.

\noindent
{\bf Acknowledgments.} Part of this paper was completed during a Research in Teams 
week at the Banff International Research Station (BIRS), and we thank BIRS for its hospitality.  
The computer algebra systems CoCoA \cite{C} and Macaulay 2 \cite{M2} were invaluable in allowing 
us to compute examples and explore conjectures. The first author is partially supported by an NSA 
Young Investigator's Grant and an Oklahoma State University Dean's Incentive Grant. The second 
author is partially supported by Board of Regents Grant LEQSF(2007-10)-RD-A-30 and Tulane's 
Research Enhancement Fund. The third author acknowledges the support provided by NSERC. 

  
\section{Betti splittings}  

We present some conditions under which we can find a Betti splitting of a monomial ideal $I$. 
Our method differs from Eliahou and Kervaire's in part because we exploit the graded 
(or multigraded) structure of $I$. Throughout, we grade the polynomial ring $R=k[x_1,\dots,x_n]$ 
either with the standard grading $\deg x_i=1$ or with the standard multigrading, in which $\deg x_i$ 
is the $i$-th unit vector $(0,\dots,0,1,0,\dots,0)$. Only Corollary~\ref{c.onevar} requires the 
multigrading; the proofs of the other results are the same in the graded case. 

Our first result shows that understanding when a monomial ideal has a Betti splitting
is equivalent to understanding when certain maps between Tor modules are the zero map.

\begin{proposition}\label{propbettisplit}
Let $I,J$, and $K$ be monomial ideals such that $I = J+K$ and $\G(I)$ is the disjoint union of $\G(J)$ and $\G(K)$;
furthermore, consider the following short exact sequence:
\begin{equation*}
(\ddagger)\hspace{1cm} 0 \rightarrow J \cap K \stackrel{\varphi}{\rightarrow} J \oplus K 
\stackrel{\psi}{\rightarrow} J+K = I \rightarrow 0
\end{equation*}
where $\varphi(f) = (f,-f)$ and $\psi(g,h) = g+h$.  Then the following are equivalent:
\begin{enumerate}
\item[(a)] $I = J+K$ is a Betti splitting.
\item[(b)] for all $i \in \N$ and all (multi)degrees $j$, the map
\[\Tor_i(k,J\cap K)_j \stackrel{\varphi_i}{\longrightarrow} \Tor_i(k,J)_j\oplus \Tor_i(k,K)_j\]
in the long exact sequence in $\Tor$ induced from  $(\ddagger)$  is the zero map.
\item[(c)] applying the mapping cone construction to $(\ddagger)$ gives a minimal free
resolution of $I$.
\end{enumerate}
\end{proposition}

\begin{proof}
$(a) \Leftrightarrow (b)$.  If $\varphi_i$ is the zero map for all $i \in \N$,
then for each $i$ and (multi)degree $j$ we have a short exact sequence
\[ 0 \longrightarrow \Tor_i(k,J)_j \oplus 
\Tor_i(k,K)_j \longrightarrow \Tor_i(k,I)_j \longrightarrow \Tor_{i-1}(k,J \cap K)_j \longrightarrow 0,\]
whence $\beta_{i,j}(I) = \beta_{i,j}(J)+\beta_{i,j}(K) + \beta_{i-1,j}(J \cap K)$, i.e.,
$I = J+K$ is a Betti splitting.  

On the other hand, suppose there is
some integer $i$ and (multi)degree $j$ such that
$\Tor_i(k,J\cap K)_j \stackrel{\varphi_i}{\longrightarrow} \Tor_i(k,J)_j\oplus \Tor_i(k,K)_j$
is not the zero map.  Assume that $i$ is the smallest such integer.
We then have the exact sequence
\[ 0 \rightarrow (\operatorname{Im} \varphi_i)_j \longrightarrow \Tor_i(k,J)_j \oplus 
\Tor_i(k,K)_j \longrightarrow \Tor_i(k,I)_j \longrightarrow \Tor_{i-1}(k,J \cap K)_j \rightarrow 0.\]
This then implies that $\beta_{i,j}(I) = \beta_{i,j}(J)+\beta_{i,j}(K) + \beta_{i-1,j}(J\cap K) -
\dim_k \operatorname{Im} \varphi_i$.  Because $\dim_k (\operatorname{Im} \varphi_i)_j > 0$,
$I = J+K$ cannot be a Betti splitting. 

$(a) \Leftrightarrow (c)$.  For any monomial ideals $I,J,$ and $K$
satisfying the hypotheses, the mapping cone construction applied to $(\ddagger)$
produces a free resolution of $I$ that is not necessarily minimal.  In particular,
the mapping cone construction implies that 
\[\beta_{i,j}(I) \leq \beta_{i,j}(J) + \beta_{i,j}(K) + \beta_{i-1,j}(J\cap K) ~~\mbox{for all $i$ and $j$,}\]
Hence, this resolution is a minimal free resolution if and only if $I = J+K$
is a Betti splitting. 
\end{proof}

When $I = J+K$ is a Betti splitting, important homological invariants of $I$ are then related to the corresponding invariants of the smaller ideals. The corollary is a direct consequence of the formulas for the Betti numbers.

\begin{corollary} \label{reg-pd} Let $I = J+K$ be a Betti splitting.  Then
\begin{enumerate}
\item[(a)] $\reg(I) = \max\{\reg(J),\reg(K),\reg(J\cap K) -1\}$, and
\item[(b)] $\operatorname{pd}(I) = \max\{\operatorname{pd}(J),\operatorname{pd}(K),
\operatorname{pd}(J\cap K)+1\}$,
\end{enumerate}
where $\reg(-)$ is the regularity, and $\operatorname{pd}(-)$ is the projective dimension.
\end{corollary}

In Eliahou and Kervaire's paper, the conditions for an E-K splitting of $I$
are used to prove that the induced map 
$\Tor_i(k,J\cap K) \stackrel{\varphi_i}{\longrightarrow} \Tor_i(k,J)\oplus \Tor_i(k,K)$ 
is the zero map for all $i$.  We can thus view the hypotheses of an E-K
splitting as one set of conditions that gives us a Betti splitting.  We are interested
in finding others;  in light of Proposition \ref{propbettisplit}, this is equivalent
to finding conditions that force the map between Tor modules to be zero. 
Our next theorem provides the basis for the other results in the section. The idea is to use 
the (multi)grading to construct $J$ and $K$ in such a way that maps between certain Tor modules 
are zero, forcing a Betti splitting.

\begin{theorem} \label{t.bettisplit}
Let $I$ be a monomial ideal in $R$, and suppose that $J$ and $K$ are monomial ideals in $R$ such 
that $\G(I)$ is the disjoint union of $\G(J)$ and $\G(K)$.  Suppose that for all $i$ and all 
(multi)degrees $j$, $\beta_{i,j}(J \cap K) > 0$ implies that $\beta_{i,j}(J)=\beta_{i,j}(K)=0$. Then 
\[ \beta_{i,j}(I) = \beta_{i,j}(J) + \beta_{i,j}(K) + \beta_{i-1,j}(J \cap K)\] for all $i$ and $j$; 
that is, $I=J+K$ is a Betti splitting.
\end{theorem}

\begin{proof}
Note that $I=J+K$, so we have a short exact sequence 
\[ 0 \longrightarrow J \cap K \longrightarrow J \oplus K \longrightarrow I \longrightarrow 0.\]
 This induces a long exact sequence in $\Tor$, which restricts to a long exact sequence of 
vector spaces upon taking (multi)graded pieces: 
\[ \cdots \longrightarrow \Tor_i(k,J \cap K)_j \longrightarrow \Tor_i(k,J)_j \oplus 
\Tor_i(k,K)_j \longrightarrow \Tor_i(k,I)_j \longrightarrow \] \[\Tor_{i-1}(k,J \cap K)_j 
\longrightarrow \Tor_{i-1}(k,J)_j \oplus \Tor_{i-1}(k,K)_j \longrightarrow \cdots\] 

Fix some $i$ and $j$, and suppose first that $\beta_{i,j}(J \cap K)= \dim_k \Tor_i(k,J \cap K)_j = 0$. 
By hypothesis, if $\beta_{i-1,j}(J \cap K) \not = 0$, then $\beta_{i-1,j}(J) = \beta_{i-1,j}(K)=0$, 
and we have a short exact sequence of vector spaces \[ 0 \longrightarrow \Tor_i(k,J)_j \oplus 
\Tor_i(k,K)_j \longrightarrow \Tor_i(k,I)_j \longrightarrow \Tor_{i-1}(k,J \cap K)_j \longrightarrow 0.\]
 Since $\dim_k$ is additive on exact sequences of vector spaces, we conclude that 
\[ \beta_{i,j}(J) + \beta_{i,j}(K) - \beta_{i,j}(I) + \beta_{i-1,j}(J \cap K) = 0\] for all $i$ and 
(multi)degrees $j$, and we have a Betti splitting.

If instead $\beta_{i-1,j}(J \cap K) = 0$, then we have an exact sequence of vector spaces 
\[ 0 \longrightarrow \Tor_i(k,J)_j \oplus \Tor_i(k,K)_j \longrightarrow \Tor_i(k,I)_j 
\longrightarrow 0,\] which again gives the desired formula for Betti numbers.

Finally, suppose $\beta_{i,j}(J \cap K) \not = 0$. Then $\beta_{i,j}(J)=\beta_{i,j}(K)=0$, and we have an 
exact sequence \[ 0 \longrightarrow \Tor_i(k,I)_j \longrightarrow \Tor_{i-1}(k,J \cap K)_j 
\longrightarrow  \Tor_{i-1}(k,J)_j \oplus \Tor_{i-1}(k,K)_j \longrightarrow \cdots\] 
If $\beta_{i-1,j}(k,J \cap K)_j=0$, then $\Tor_i(k,I)_j=0$, so $\beta_{i,j}(I)=0$, and the formula 
holds. Alternatively, if $\beta_{i-1,j}(k,J \cap K)_j \not = 0$, then our hypothesis implies 
that  $\Tor_{i-1}(k,J)_j=\Tor_{i-1}(k,K)_j=0$, and $\beta_{i,j}(I)=\beta_{i-1,j}(J \cap K)$, proving the 
Betti number formula since we are assuming $\beta_{i,j}(J)=\beta_{i,j}(K)=0$.
\end{proof}

Of course, if the conditions of Theorem~\ref{t.bettisplit} hold for all multidegrees $j$, 
then we have the Betti splitting formula for both the graded Betti numbers and total Betti 
numbers of $I$ in terms of those of $J$, $K$, and $J \cap K$. Additionally, we have an 
easy corollary when $J$ and $K$ both have linear resolutions.

\begin{corollary} \label{c.bothlinear}
Let $I$ be a monomial ideal in $R$, and suppose that $J$ and $K$ are monomial ideals in 
$R$ such that $\G(I)$ is the disjoint union of $\G(J)$ and $\G(K)$. If both $J$ and $K$ 
have linear resolutions, then $I=J+K$ is a Betti splitting.
\end{corollary}

\begin{proof}
We may assume that the degree of any monomial in $\G(J)$ is $d_J$, and the degree of any 
monomial in $\G(K)$ is $d_K$. Since $\G(I)$ is the disjoint union of $\G(J)$ and $\G(K)$, 
$\G(J \cap K)$ is comprised of monomials of degree greater than $d=\max(d_J,d_K)$. Since 
$\reg(J) \le d$ and $\reg(K) \le d$, but $J \cap K$ is generated in degrees at least as 
high as $d+1$, we conclude that for all $i$ and all (multi)degrees $j$, $\beta_{i,j}(J \cap K) > 0$ 
implies that $\beta_{i,j}(J)=\beta_{i,j}(K)=0$. Thus by Theorem~\ref{t.bettisplit}, $I=J+K$ 
is a Betti splitting.
\end{proof}

Corollary~\ref{c.bothlinear} allows us some insight into the example in Eliahou and 
Kervaire's paper that motivated our work.

\begin{example} \label{e.ek}
Let $S=k[x_1,\dots,x_5]$, and let \[I=(x_1x_2x_3,x_1x_3x_5,x_1x_4x_5,x_2x_3x_4,x_2x_4x_5).\] 
Eliahou and Kervaire note in their paper that there exists no E-K splitting of $I$. This is 
relatively easy to check; for example, suppose $J=(x_1x_2x_3,x_1x_3x_5,x_1x_4x_5)$, and 
$K=(x_2x_3x_4,x_2x_4x_5)$. Then $J \cap K = (x_1x_2x_3x_4,x_1x_2x_4x_5)$. In order to map 
$\G(J \cap K)$ to $\G(J) \times \G(K)$, we have to send $x_1x_2x_3x_4$ to $(x_1x_2x_3,x_2x_3x_4)$, 
and $x_1x_2x_4x_5$ must map to $(x_1x_4x_5,x_2x_4x_5)$. (Here, these are elements of 
$\G(J) \times \G(K)$, not ideals.) But then the least common multiple of the first components 
is $\lcm(x_1x_2x_3,x_1x_4x_5)=x_1x_2x_3x_4x_5$, which does not strictly divide 
$\lcm(x_1x_2x_3x_4,x_1x_2x_4x_5)$.

However, $J$ and $K$ both have linear resolutions, and so by Corollary~\ref{c.bothlinear}, 
$I=J+K$ is a Betti splitting.
\end{example}

The partitioning of the generators in Example~\ref{e.ek} has a particularly convenient 
form that is useful for investigating monomial ideals in combinatorics.

\begin{definition} \label{d.onevar}
Let $I$ be a monomial ideal in $R=k[x_1,\dots,x_n]$. Let $J$ be the ideal generated by all 
elements of $\G(I)$ divisible by $x_i$, and let $K$ be the ideal generated by all other 
elements of $\G(I)$. We call $I=J+K$ an {\boldmath $x_i$}\textbf{-partition} of $I$. If $I=J+K$ 
is also a Betti splitting, we call $I=J+K$ an {\boldmath $x_i$}\textbf{-splitting}.
\end{definition}

\begin{corollary} \label{c.onevar}
Let $I=J+K$ be an $x_i$-partition of $I$ in which all elements of $J$ are divisible by $x_i$. 
If $\beta_{i,j}(J \cap K) > 0$ implies that $\beta_{i,j}(J)=0$ for all $i$ and multidegrees $j$, 
then $I=J+K$ is a Betti splitting. In particular, if the minimal graded free resolution of $J$ is 
linear, then $I=J+K$ is a Betti splitting. 
\end{corollary}

\begin{proof}
Note that all elements of both $J$ and $J \cap K$ are divisible by $x_i$, so all the multigraded 
Betti numbers of $J$ and $J \cap K$ occur in degrees divisible by $x_i$, and none of the 
multigraded Betti numbers of $K$ do. Therefore $\beta_{i,j}(J \cap K) > 0$ implies that 
$\beta_{i,j}(K)=0$ for all $i$ and multidegrees $j$; since the same implication holds for the 
multigraded Betti numbers of $J$ by hypothesis on the graded resolution of $J$, the first statement 
follows from Theorem~\ref{t.bettisplit}.

For the last statement, assume that $J$ has a linear resolution.  Then 
$J \cap K$ is generated in higher degrees than $J$, and therefore 
$\beta_{i,j}(J \cap K) > 0$ implies that $\beta_{i,j}(J)=0$ for all $i$ and multidegrees $j$.
\end{proof}

One class of ideals that is important in computational commutative algebra is that of 
stable ideals; the Borel-fixed ideals in characteristic zero are precisely the strongly 
stable ideals, a subclass. Eliahou and Kervaire point out using an E-K splitting argument 
that all stable ideals have an $x_1$-splitting, though using this is likely not more 
efficient for computing the Betti numbers of stable ideals than simply relying on the 
formulas from the standard Eliahou-Kervaire resolution. Unfortunately, our 
Theorem~\ref{t.bettisplit} does not prove that all stable ideals have an $x_1$-splitting 
because there could be $i$ and $j$ such that $\beta_{i,j}(J)$ and $\beta_{i,j}(J \cap K)$ are both nonzero. 
For example, if $I$ is the smallest Borel-fixed ideal in $S=k[x_1,\dots,x_6]$ with $x_1x_6^3$ 
and $x_3^2x_6$ as minimal generators (in Macaulay 2, one obtains this with the command 
{\tt borel monomialIdeal(x\_1*x\_6\verb!^!3,x\_3\verb!^!2*x\_6)}), let $I=J+K$ be an 
$x_1$-partition, and let $j$ correspond to the multidegree of $x_1x_2x_3x_4x_5x_6$. 
Then $\beta_{2,j}(J)$ and $\beta_{2,j}(J \cap K)$ are both nonzero. Thus the E-K splittings and 
our Betti splittings each apply to some ideals to which the other does not.

  
\section{Applications to edge ideals}  

We apply the results of the previous section to some combinatorial settings. We focus on 
ideals associated to graphs. Let $G = (V,E)$ be a simple graph (no loops 
or multiple edges) on the vertices $V = \{x_1,\ldots,x_n\}$ and edge set
$E$.  By identifying the variables 
of the polynomial ring $R = k[x_1,\ldots,x_n]$ with the vertices of $V$, we can associated 
to $G$ a square-free monomial ideal $I(G) = (\{x_ix_j ~| \{x_i,x_j\} \in E\}),$
called the {\bf edge ideal} of $G$. 

One natural way to try to split an edge ideal $I(G)$ is to seek an $x_i$-splitting. 
Following \cite{HVTsplit}, if $x_i$ is a vertex of $G$ that is not isolated and such that 
$G \setminus \{x_i\}$ is not a graph of isolated vertices, we call $x_i$ a \textbf{splitting vertex} 
of $G$. (Isolated vertices do not affect the Betti numbers of $I(G)$, and if 
$G \setminus \{x_i\}$ consists only of isolated vertices, the Betti numbers of $I(G)$ are 
easy to compute since $G$ is a complete bipartite graph plus possibly some isolated vertices.) 
Using Corollary~\ref{c.onevar}, we recover \cite[Theorem 4.2]{HVTsplit}, which was 
instrumental in \cite{HVTsplit} in unifying a number of previous works on resolutions of edge 
ideals, in one sentence.

\begin{corollary} \label{c.edgesplit} \cite[Theorem 4.2]{HVTsplit}
Let $G$ be a simple graph with edge ideal $I(G)$ and splitting vertex $x_i$. Let $J$ be the 
ideal generated by all elements of $\G(I)$ divisible by $x_i$, and $K$ be generated by 
$\G(I(G)) \setminus \G(J)$. Then $I(G)=J+K$ is an $x_i$-splitting. 
\end{corollary}

\begin{proof}
$J$ is $x_i$ times an ideal generated by a subset of the variables, so it has a linear 
resolution, and the result follows from Corollary~\ref{c.onevar}.
\end{proof}

\begin{remark} \label{r.hypergraphs}
One can generalize Corollary~\ref{c.edgesplit} to the setting of $d$-uniform properly-connected triangulated hypergraphs by using \cite[Theorem 6.8]{HVThyper} to prove that the ideal $J$, which consists of all hyperedges containing some fixed $x_i$, has a linear resolution. 
\end{remark}

Our second combinatorial application is a recursive formula for the graded Betti numbers 
of the cover ideal of a Cohen-Macaulay bipartite graph. We begin by introducing some 
terminology and Herzog and Hibi's classification of such graphs.

We call a graph $G$ a {\bf Cohen-Macaulay graph} if the ring $R/I(G)$ is Cohen-Macaulay. 
Identifying classes of Cohen-Macaulay graphs is a topic of much interest 
\cite{FH,FV,HH,V}. A graph-theoretic description of Cohen-Macaulay bipartite 
graphs was found by Herzog and Hibi \cite{HH}. We say a graph is {\bf bipartite} 
if there is a bipartition of $V = V_1 \cup V_2$ such that every edge of $G$ has one 
vertex in $V_1$ and the other in $V_2$.
Herzog and Hibi then proved:

\begin{theorem}\label{cmbipartite}
{\cite[Theorem 3.4]{HH}}  Let $G$ be a bipartite graph with bipartition 
$V = \{x_1,\ldots,x_n\} \cup \{y_1,\ldots,y_m\}$.  Then $G$ is Cohen-Macaulay if and only 
if $n=m$, and there is a labeling such that
\begin{enumerate}
\item[(a)] $\{x_i,y_i\} \in E$ for $i=1,\ldots,n$,
\item[(b)] whenever $\{x_i,y_j\} \in E$, then $i \leq j$, and 
\item[(c)] whenever $\{x_i,y_j\}$ and $\{x_j,y_k\}$ are edges of $G$ with $i < j < k$,
then $\{x_i,y_k\} \in E$.
\end{enumerate}
\end{theorem}

The following observations about Cohen-Macaulay bipartite graphs shall be useful. 
First, by Theorem \ref{cmbipartite}, the vertex $x_n$ must have degree one and is 
only adjacent to $y_n$. Second, use the notation $N(y_n)$ to
denote the {\bf neighbors} of $y_n$; that is, \[N(y_n) = \{z \in V ~|~ \{z,y_n\} \in E\}.\] 
Because $G$ is bipartite, $N(y_n) = \{x_{i_1},\ldots,x_{i_s},x_n\}$ for some $x_{i_j} \in \{x_1,\ldots,x_n\}$.

\begin{lemma}  \label{cmsubgraphs}
Let $G$ be a Cohen-Macaulay bipartite graph.  Then
\begin{enumerate}
\item[(a)] $G \setminus \{y_n,x_n\}$ is a Cohen-Macaulay bipartite graph.
\item[(b)] If $N(y_n) =  \{x_{i_1},\ldots,x_{i_s},x_n\}$, then 
$G \setminus \{x_{i_1},y_{i_1},\ldots,x_{i_s},y_{i_s},x_n,y_n\}$ is a Cohen-Macaulay bipartite graph.
\end{enumerate}
\end{lemma}

\begin{proof}  For (a), by $G \setminus \{y_n,x_n\}$ we mean the graph with vertices 
$x_n$, $y_n$, and all the edges adjacent to these vertices removed. Note that this is 
the same graph as $G \setminus \{y_n\}$, except this second graph has an isolated vertex, 
namely $x_n$.  It is straightforward to check that the conditions of Theorem \ref{cmbipartite} 
still hold for $G \setminus \{y_n,x_n\}$.

For (b), first note that when we remove $y_n$ and its neighbors $N(y_n)$ from $G$, the 
vertices $\{y_{i_1},\ldots,y_{i_s}\}$ must all be isolated vertices in the graph $G \setminus 
\{y_n \cup N(y_n)\}.$  Indeed, suppose that there is an edge in $G \setminus \{y_n \cup N(y_n)\}$ 
that contains $y_{i_j} \in  \{y_{i_1},\ldots,y_{i_s}\}$. Because $G$ is bipartite (and thus, so is 
$G \setminus \{y_n \cup N(y_n)\}$), 
this edge must have form $\{x_k,y_{i_j}\}$.  But by Theorem \ref{cmbipartite}, we must have 
$k \leq i_j$.  However, we cannot have $k = i_j$ since $x_{i_j}$ has been removed. But then 
in $G$ we have edges $\{x_k,y_{i_j}\}$ and $\{x_{i_j},y_n\}$, and hence, by Theorem 
\ref{cmbipartite}, the edge $\{x_k,y_n\}$ is also in $G$.   Hence,
$x_k \in N(y_n)$, contradicting the fact that $x_k \in G \setminus \{y_n \cup N(y_n)\}$. 
So, removing the isolated vertices of $G \setminus \{y_n \cup N(y_n)\}$ gives us the graph 
$G \setminus \{x_{i_1},y_{i_1},\ldots,x_{i_s},y_{i_s},x_n,y_n\}$.  Again, one can check that the 
conditions of Theorem \ref{cmbipartite} hold for this graph.
\end{proof}

A subset $W \subseteq V$ is called a {\bf vertex cover} if every edge $e = \{u,v\} \in E$ 
has non-empty intersection with $W$.  We call $W$ a {\bf minimal vertex cover} if $W$ is a 
vertex cover, but no proper subset of $W$ is a vertex cover.  Attached to $G$ is another 
square-free monomial ideal, called the {\bf cover ideal}, defined by
\[I(G)^{\vee} = (\{ x_{i_1}\cdots x_{i_s} ~|~ W = \{x_{i_1},\ldots,x_{i_s}\} ~~\mbox{is a minimal 
vertex cover of $G$}\}).\] Note that the cover ideal is  the Alexander dual of the edge 
ideal $I(G)$; this explains our use of the notation $I(G)^{\vee}$. To compute even the $0$-th Betti number of $I(G)^{\vee}$ from $G$ itself is difficult since it requires knowing how many minimal vertex covers $G$ has.

\begin{remark}  \label{isolatedvertex}
If $z$ is an isolated vertex of $G$, then the cover ideals of $G$ and $G \setminus \{z\}$ 
are exactly the same, assuming we consider both as ideals of the (larger) ring inside which 
$I(G)^{\vee}$ lives. The proof of Lemma \ref{cmsubgraphs} then implies that 
$I(G\setminus \{y_n\})^{\vee} = I(G\setminus \{y_n,x_n\})^{\vee}$,
and $I(G\setminus \{y_n \cup N(y_n)\})^{\vee} = 
I(G\setminus \{x_{i_1},y_{i_1},\ldots,x_{i_s},y_{i_s},x_n,y_n\})^{\vee}$. 
\end{remark}

\begin{lemma} \label{presplit}
Let $G$ be a Cohen-Macaulay bipartite graph.  Let $y_n$ be the unique vertex adjacent to
$x_n$, and suppose that $N(y_n) = \{x_{i_1},\ldots,x_{i_s},x_n\}$.  Then
\[I(G)^{\vee} = y_nI(G\setminus\{y_n\})^{\vee} + x_{i_1}\cdots x_{i_s}x_nI(G\setminus\{y_n \cup N(y_n)\})^{\vee}.\]
\end{lemma}

\begin{proof}  In order to cover the edge $\{x_n,y_n\}$, every vertex cover must contain
at least one of $x_n$ and $y_n$.  In fact, any minimal vertex cover of $G$ must contain exactly one 
of $x_n$ or $y_n$; if $W$ is any vertex cover that contains both $x_n$ and $y_n$, then 
$W \setminus \{x_n\}$ remains a vertex cover of $G$. Thus, if $m$ is a minimal generator of 
$I(G)^{\vee}$, it is divisible by exactly one of $x_n$ and $y_n$.  If $y_n|m$, then $\frac{m}{y_n}$ 
must correspond to a cover of $G \setminus \{y_n\}$, and hence, 
$\frac{m}{y_n} \in I(G\setminus \{y_n\})^{\vee}$. If $x_n|m$, then $y_n \nmid m$, so $x_{i_1},\ldots,x_{i_s}$ 
must also divide $m$ so that
all edges adjacent to $y_n$ are covered.  It then follows that $\frac{m}{x_{i_1}\cdots x_{i_s}x_n}
\in I(G \setminus \{y_n \cup N(y_n)\})^{\vee}$.

Conversely, it is easy to see that minimal generators of $y_nI(G\setminus\{y_n\})^{\vee}$
and $x_{i_1}\cdots x_{i_s}x_nI(G\setminus\{y_n \cup N(y_n)\})^{\vee}$ correspond to vertex covers
 of $G$.
\end{proof}

We need one more result, a theorem due to Eagon and Reiner:

\begin{theorem}\cite[Theorem 3]{ER}  \label{eagonreiner} 
Let $I$ be a square-free monomial ideal.  Then $R/I$ is Cohen-Macaulay if and only if the Alexander 
dual $I^{\vee}$ has a linear resolution.
\end{theorem}

Because we are interested in the resolution of $I(G)^{\vee}$ when $G$ is a Cohen-Macaulay
bipartite graph, Theorem \ref{eagonreiner} implies that $I(G)^{\vee}$ has a linear resolution.
Now $I(G)^{\vee}$ is generated by monomials of degree $n$.  So, $\beta_{i,j}(I(G)^{\vee}) = 0$
for all $j \neq n+i$.  In this case, the $i$-th total Betti number of $I(G)$ equals the 
$\beta_{i,n+i}(I(G)^{\vee})$. Thus it suffices to find the $i$-th total Betti numbers. As we 
show below, these can be computed recursively.  The formula is based upon the fact that 
we can find a Betti splitting of the monomial ideal $I(G)^{\vee}$.

\begin{theorem} \label{rescmbipartite} Let $G$ be a Cohen-Macaulay bipartite graph.  Suppose that 
$y_n$ is the unique vertex adjacent to $x_n$ and that $N(y_n) = \{x_{i_1},\ldots,x_{i_s},x_n\}$.
Then, \[\beta_i(I(G)^{\vee}) = \beta_i(I(G')^{\vee}) + \beta_i(I(G'')^{\vee}) + \beta_{i-1}(I(G'')^{\vee})
~~\mbox{for all $i \geq 0$}\] where $G' = G \setminus \{x_n,y_n\}$ and 
$G'' = G \setminus  \{x_{i_1},y_{i_1},\ldots,x_{i_s},y_{i_s},x_n,y_n\}$, both of which are Cohen-Macaulay bipartite graphs.
\end{theorem}

\begin{proof}
By Lemma \ref{presplit},
\begin{equation}\label{splitcover} 
I(G)^{\vee} = y_nI(G\setminus\{y_n\})^{\vee} + x_{i_1}\cdots x_{i_s}x_nI(G\setminus\{y_n \cup N(y_n)\})^{\vee}.
\end{equation}
This is a $y_n$-partition of $I(G)^{\vee}$. By Remark \ref{isolatedvertex}, $I(G\setminus \{y_n\})^{\vee}
= I(G')^{\vee}$, and $I(G\setminus\{y_n \cup N(y_n)\})^{\vee} = I(G'')^{\vee}$. Moreover, by
 Lemma \ref{cmsubgraphs}, $G'$ and $G''$ are both Cohen-Macaulay. By Theorem \ref{eagonreiner}, 
$I(G')^{\vee}$ has a linear resolution,
and thus, so does $y_nI(G')^{\vee}$. It then follows from Corollary~\ref{c.onevar} that (\ref{splitcover})
is a Betti splitting of $I(G)^{\vee}$.  Since we are only interested in the total Betti numbers, we get
\[
\beta_{i}(I(G)^{\vee}) = \beta_{i}( y_nI(G')^{\vee})+
\beta_{i}(x_{i_1}\cdots x_{i_s}x_nI(G'')^{\vee})+ \beta_{i-1}( y_nI(G')^{\vee} \cap x_{i_1}\cdots x_{i_s}x_nI(G'')^{\vee}).
\]
Note that $\beta_i(y_nI(G')^{\vee}) = \beta_i(I(G')^{\vee})$ and 
$\beta_i(x_{i_1}\cdots x_{i_s}x_nI(G'')^{\vee}) = \beta_i(I(G'')^{\vee})$.  The proof will
then be complete once we prove the claim below since the claim implies
that the right-most expression in the above formula equals $\beta_{i-1}(I(G'')^{\vee})$.

\noindent
{\it Claim.}  $y_nI(G')^{\vee} \cap x_{i_1}\cdots x_{i_s}x_nI(G'')^{\vee} = y_nx_{i_1}\cdots x_{i_s}x_nI(G'')^{\vee}$.
\vspace{.25cm}

\noindent
{\it Proof of the Claim.}  Note that 
$y_nx_{i_1}\cdots x_{i_s}x_nI(G'')^{\vee} \subseteq x_{i_1}\cdots x_{i_s}x_nI(G'')^{\vee}$.  Furthermore, if $m$ is a 
generator of  $y_nx_{i_1}\cdots x_{i_s}x_nI(G'')^{\vee}$,
then $m$ is also in $y_nI(G')^{\vee}$ since $m = y_nx_{i_1}\cdots x_{i_s}x_nm'$ and $x_{i_1}\cdots x_{i_s}m'$ is 
a cover of $G'$.  This gives us
$$y_nI(G')^{\vee} \cap x_{i_1}\cdots x_{i_s}x_nI(G'')^{\vee} \supseteq y_nx_{i_1}\cdots x_{i_s}x_nI(G'')^{\vee}.$$

On the other hand, observe that a minimal generator of
 $y_nI(G')^{\vee} \cap x_{i_1}\cdots x_{i_s}x_nI(G'')^{\vee}$ has the form $\lcm(m_1,m_2)$, where 
$m_1$ is a generator of $y_nI(G')^{\vee}$ and $m_2$ is a generator
of $x_{i_1}\cdots x_{i_s}x_nI(G'')^{\vee}$. We can write $m_1 = y_nm'_1$ and $m_2 = x_{i_1}\cdots x_{i_s}x_nm'_2$ 
where $m'_1 \in I(G')^{\vee}$ and $m'_2 \in I(G'')^{\vee}$. Now, it is easy to see that 
$y_nx_{i_1}\cdots x_{i_s}x_nm'_2$ divides $\lcm(m_1,m_2)$. Thus, 
$\lcm(m_1,m_2) \in y_nx_{i_1}\cdots x_{i_s}x_nI(G'')^{\vee}$, and we have the other containment. 
\end{proof}

\begin{remark}
Because $G'$ and $G''$ in Theorem~\ref{rescmbipartite} are Cohen-Macaulay and bipartite, we can compute the Betti numbers of $I(G)^{\vee}$ recursively.  These graded Betti
numbers do not depend upon the characteristic of the field.
\end{remark}

We can now easily recover a special case of Kummini's \cite[Theorem 1.1]{K}. Recall that two edges $\{x_{i_1},x_{i_2}\}$ and $\{x_{i_3},x_{i_4}\}$ of a graph $G$ are said to be \textbf{3-disjoint} (or \textbf{disconnected}) if the induced subgraph of $G$ on $\{x_{i_1},\dots,x_{i_4}\}$ consists of two disjoint edges (i.e, it is the complement of a 4-cycle). (See \cite[Definition 6.3]{HVThyper}.) For a graph $G$, write $a(G)$ for the maximum size of a set of pairwise 3-disjoint edges in $G$; this is the largest number of edges in an induced subgraph of $G$ in which each connected component is an edge.

\begin{corollary} \label{pdim-cmbipartite}
Let $G$ be a Cohen-Macaulay bipartite graph. Then $\operatorname{pd}(I(G)^{\vee})=\reg(R/I(G)) = a(G)$.
\end{corollary}

\begin{proof}
By Alexander duality of square-free monomial ideals, we have  $\operatorname{pd}(I(G)^{\vee})=\reg(R/I(G))$. By \cite[Lemma 2.2]{Katzman}, $\operatorname{reg}(R/I(G))\ge a(G)$. Let $G'$ and $G''$ be as in Theorem~\ref{rescmbipartite}. Clearly $a(G) \ge a(G')$. Additionally, $a(G) \ge 1+a(G'')$ since we can add the edge $\{x_n,y_n\}$ to any set of pairwise 3-disjoint edges of $G''$ to obtain a corresponding set of edges of $G$.  Because Theorem~\ref{rescmbipartite} gives a Betti splitting, by Corollary~\ref{reg-pd}, \[\operatorname{pd}(I(G)^{\vee}) = \max \{\operatorname{pd} (I(G')^{\vee}), \operatorname{pd}(I(G'')^{\vee})+1 \} = \max\{a(G'),a(G'')+1\} \le a(G),\] where the second equality follows by induction. Hence $\reg(R/I(G)) = \operatorname{pd}(I(G)^{\vee})=a(G)$.
\end{proof}

The connection between regularity and disconnected edges first appeared in Zheng's paper \cite{Zheng} and was extended to the case of chordal graphs in \cite[Corollary 6.9]{HVThyper}. Kummini showed that the conclusion to Corollary \ref{pdim-cmbipartite} is still true if $G$ is a bipartite graph whose edge ideal is unmixed.  The minimal resolutions of bipartite graphs whose edge ideals are unmixed were also studied by Mohammadi and Moradi \cite{MM}; in this paper, the regularity of $I(G)$ is given in terms of a lattice constructed from the minimal vertex covers of $G$.

  
\section{Observations from computational experiments and splittings in positive characteristic}  
 
We ran a large number of computational tests in Macaulay 2 when working on this project, trying to
 understand convenient combinatorial or algebraic conditions under which a monomial ideal has a 
Betti splitting. We were particularly interested in finding $x_i$-splittings for monomial ideals, 
and the tests we ran indicate that it is extremely rare for a monomial ideal to have no 
$x_i$-splitting. Out of tens of thousands of tests in Macaulay 2, we found only a handful of 
examples. This suggests that the notion of an $x_i$-splitting can be particularly helpful when 
investigating Betti numbers, particularly in inductive arguments in which one inducts on the 
dimension of the ring (in the combinatorial setting, on the number of vertices of a graph, hypergraph, or 
simplicial complex).  

We mention a few examples of ideals with no $x_i$-splitting.  We wish to highlight
an apparent connection between monomial
ideals with no $x_i$-splitting and monomial ideals whose resolutions
are characteristic-dependent.

\begin{example}
One particularly interesting example is the 
Stanley-Reisner ideal of a triangulation of the real projective plane:
\[ I = (x_1x_2x_4,x_1x_2x_6,x_1x_3x_5,x_1x_3x_4,x_1x_5x_6,x_2x_4x_5,x_2x_3x_6,x_2x_3x_5,x_3x_4x_6,x_4x_5x_6).\]
 Consider the $x_1$-partition; the behavior is the same for the other variables. Assume the 
characteristic of $k$ is not 2. We have \[J =(x_1x_2x_4,x_1x_2x_6,x_1x_3x_5,x_1x_3x_4,x_1x_5x_6), \text{ and}\] 
\[K =(x_2x_4x_5,x_2x_3x_6,x_2x_3x_5,x_3x_4x_6,x_4x_5x_6).\] Additionally, $J \cap K = x_1K$ (just as with an
 $x_1$-partition of a stable ideal). The minimal resolutions of $J$, $K$, and $J \cap K$ are all 
\[ 0 \longrightarrow R^1 \longrightarrow R^5 \longrightarrow R^5 \longrightarrow L \longrightarrow 0,\] 
(where $L$ is standing in for $J,K$, or $J \cap K$) and the minimal 
resolution of $I$ is 
\[ 0 \longrightarrow R^6 \longrightarrow R^{15} \longrightarrow R^{10} \longrightarrow I \longrightarrow 
0.\] 
If this were a Betti splitting of $I$, then $I$ would have minimal resolution 
\[ 0 \longrightarrow R^1 \longrightarrow R^7 \longrightarrow R^{15} \longrightarrow R^{10} \longrightarrow I \longrightarrow 0,\] 
so $I=J+K$ is not an $x_1$-splitting.

However, when the characteristic of $k$ is 2, the minimal resolution of $I$ has extra syzygies in the 
multidegree $j$ corresponding to $x_1x_2x_3x_4x_5x_6$; $\beta_{2,j}(I)=\beta_{3,j}(I)=1$. Now the fact that 
$\beta_{2,j}(J)=\beta_{2,j}(J \cap K)=1$ is no longer a problem, and $I=J+K$ is an $x_1$-splitting; in
 fact, $I$ admits an $x_i$-splitting for all $x_i$. This is perhaps a rare example of nicer homological 
behavior in an exceptional positive characteristic than in characteristic zero.
\end{example}

\begin{example}
There are other examples of ideals with no $x_i$-splitting in almost all characteristics but an 
$x_i$-splitting in characteristic 2. In seven variables, we have 
\[I'=(x_2x_6x_7,x_1x_6x_7,x_4x_5x_7,x_3x_4x_7,x_1x_4x_7,x_2x_3x_7, x_1x_3x_7,x_4x_5x_6,x_2x_5x_6,x_1x_5x_6,\]
\[x_3x_4x_6,x_2x_4x_6, x_2x_4x_5, x_2x_3x_5,x_1x_3x_5,x_1x_3x_4,x_1x_2x_4),\] which has an extra syzygy in 
the multidegree corresponding to $x_1x_2x_3x_4x_5x_6x_7$ in characteristic two. $I'$ has an 
$x_4$-splitting in characteristic 2 (and no other $x_i$-splitting) but no $x_i$-splitting in other characteristics.
\end{example}

Not all examples of monomial ideals with no $x_i$-splitting have characteristic-dependent resolution. 
If $M$ is the ideal generated by the minimal generators of the ideal $I'$ above 
except for $x_1x_3x_4$, then the Betti 
numbers of $M$ do not depend on the characteristic of $k$. However, $M$ has no $x_i$-splitting. 
Moreover, not all ideals with characteristic-dependent resolutions fail to have an $x_i$-splitting 
in some characteristic. Katzman constructed a number of examples of edge ideals of graphs whose 
Betti numbers depend on the ground field \cite{Katzman}, but by Corollary~\ref{c.edgesplit}, these 
edge ideals all have $x_i$-splittings. Nevertheless, since almost every monomial ideal we tested has 
an $x_i$-splitting, and the exceptions are related to ideals whose resolutions are characteristic-dependent, 
we conclude by asking the following very broad question:

\begin{question} \label{q.characteristic}
Is the class of monomial ideals with no $x_i$-splitting somehow connected to the class of monomial ideals 
whose resolutions depend upon the ground field?
\end{question}


\end{document}